\documentclass[12pt,leqno]{amsart}

\usepackage{amsmath,amsfonts,amssymb,amsthm,epsfig,mathtools}

 \allowdisplaybreaks

\def\R{\mathbb R}

\def\diam{\textrm{diam}}
\def\grad{\nabla}

\newtheorem{Thm}{Theorem}
\newtheorem{Prop}{Proposition}
\newtheorem{Lem}{Lemma}
\newtheorem*{Thm*}{Theorem}
\newtheorem{Rem}{Remark}
\newtheorem{Def}{Definition}

\newtheorem{Cor}{Corollary}

\begin{document}

\title[The Neumann eigenvalue problem for the $\infty$-Laplacian]{The Neumann eigenvalue problem for the $\infty$-Laplacian}

\author{L.~Esposito}\address{Dipartimento di Matematica e Informatica, via Ponte Don Melillo, 84084 Fisciano (SA)}\email{luesposi@unisa.it}
\author{B.~Kawohl}\address{Mathematisches Institut, Universit\"at zu K\"oln,
50923 K\"oln}\email{kawohl@math.uni-koeln.de}
\author{C.~Nitsch}\address{Dipartimento di Matematica e Applicazioni, via Cintia, 80126 Napoli}\email{c.nitsch@unina.it}
\author{C.~Trombetti}\address{Dipartimento di Matematica e Applicazioni, via Cintia, 80126 Napoli}\email{cristina@unina.it}

\begin{abstract}The first nontrivial eigenfunction of the Neumann eigenvalue problem for the $p$-Laplacian, suitable normalized, converges to a viscosity solution of an eigenvalue problem for the $\infty$-Laplacian. We show among other things that the limiting eigenvalue, at least for convex sets, is in fact the first nonzero eigenvalue of the limiting problem. We then derive a number consequences, which are nonlinear analogues of well-known inequalities for the linear (2-)Laplacian.
\end{abstract}
\maketitle
{\small

\keywords{\noindent {\bf Keywords:}
Neumann eigenvalues, viscosity solutions, infinity Laplacian }

\smallskip

\subjclass{\noindent {\bf 2010 MSC:} 35P30, 35P15, 35J72, 35D40, 35J92, 35J70}

\section{Introduction and statements}

In this paper we study the $\infty$-Laplacian eigenvalue problem under Neumann boundary conditions
\begin{equation}\label{infn}
\begin{dcases}
{\rm {min}}  \{| \grad u| - \Lambda u, -\Delta_{\infty} u\} = 0 & \mbox{in}\ \{u >0\}\cap\Omega\\
{\rm {max}} \{-| \grad u|  - \Lambda u , -\Delta_{\infty} u\}=0 & \mbox{in}\ \{u <0\}\cap\Omega\\
-\Delta_{\infty} u =0,& \mbox{in} \ \{u =0\}\cap\Omega\\
\frac{\partial u}{\partial \nu}=0 & \mbox{on} \ \partial\Omega.\\
\end{dcases}
\end{equation}
A solution $u$ to this problem has to be understood in the viscosity sense, and the Neumann eigenvalue $\Lambda$ is some nonnegative real constant. For $\Lambda=0$ problem (\ref{infn}) has constant solutions. We consider those as trivial. Our main result is 

\begin{Thm}\label{thm_main}
Let $\Omega$ be a smooth bounded open convex set in $\R^n$ then a necessary condition for the existence of nonconstant continuous solutions u to \eqref{infn} is
\begin{equation}\label{mainresult}
\Lambda\geq\Lambda_\infty:=\frac2{\diam(\Omega)}.\end{equation}
Here $diam(\Omega)$ denotes the diameter of $\Omega$.
Moreover problem \eqref{infn} admits a Lipschitz solution when $\Lambda=\frac2{\diam(\Omega)}$. 
\end{Thm}

If $\Omega$ is merely bounded, connected and has Lipschitz boundary, then the notion of diameter can be generalized as in Definition \ref{intrinsic}. In that case solutions of (\ref{infn}) exist, see Section \ref{limsec} or \cite{RS}. However, it is still unclear whether $\Lambda_\infty$ is always the first eigenvalue.

Theorem \ref{thm_main} has a number of interesting consequences, one of which we list right here. By the isodiametric inequality we may conclude
\begin{Cor} If $\Omega^*$ denotes the ball of same volume as $\Omega$, then the Szeg\"o-Weinberger inequality $\Lambda_\infty(\Omega)\leq\Lambda_\infty(\Omega^*)$ holds.
\end{Cor}
For the case of the ordinary Laplacian $(p=2)$ this result was shown in \cite{S} and \cite{W}. For the 1- Laplacian case and convex plane $\Omega$ we refer to  \cite{EFKNT}.
 While the Faber-Krahn inequality $\lambda_p(\Omega^*)\leq \lambda_p(\Omega)$ holds for any $p$, the Szeg\"o-Weinberger inequality has resisted attempts to be generalized to general $p$, and for general $p$ we are unaware of any results in this direction. \\%Moreover, explicit knowledge of $\lambda_p(\Omega^*)$ or $\Lambda_p(\Omega^*)$ seems to exist only if $p\in\{1, 2, \infty\}$. \\
The reason why we call problem \eqref{infn}  $\infty$-Laplacian eigenvalue problem under Neumann boundary conditions is that \eqref{infn} can be derived as the limit $p\to \infty$ of  Neumann eigenvalue problems for the $p$-Laplacian
\begin{equation}\label{neup}
\begin{dcases}
-\Delta_p u=\Lambda_p^p |u|^{p-2}u & \mbox{in }\Omega\\
|\grad u|^{p-2}\frac{\partial u}{\partial n}=0 & \mbox{on }\partial\Omega,
\end{dcases}
\end{equation}
whenever $\Omega$ is a bounded open Lipschitz set of $\R^n$. %This yields to a degenerate eigenvalue problem for the so called infinity Laplacian. The study of $\infty$-Laplacian goes back to the pioneering works \cite{} and has many applications \cite{}.... \\

For the Dirichlet $p$-Laplacian eigenvalue problem on open bounded sets $\Omega\subset\R^n$ 
\begin{equation}\label{dirp}
\begin{dcases}
-\Delta_p v=\lambda_p^p |v|^{p-2}v & \mbox{in }\Omega\\         
v=0 & \mbox{on }\partial\Omega,
\end{dcases}
\end{equation} 
the same limit was studied by Juutinen, Lindqvist and Manfredi in \cite{JLM,JL}. They formulate and fully investigate the so-called Dirichlet $\infty$-Laplacian eigenvalue problem employing the notion of viscosity solutions. Recall for instance that, when $\lambda_p$ denotes for all $p\ge 1$ the first nontrivial eigenvalue of \eqref{dirp}, the limit yields

\[\lim_{p\to\infty}\lambda_p=\lambda_\infty:=\frac1{R(\Omega)},\]
where $R(\Omega)$ denotes inradius, i.e.  the radius of the largest ball contained in $\Omega$. Moreover, they identify the limiting eigenvalue problem as
\begin{equation}\label{dirinf}
\begin{dcases}
\min\{|\grad v|-\lambda v, -\Delta_\infty v\}=0 & \mbox{in }\Omega\\
v=0 & \mbox{on }\partial\Omega,
\end{dcases}
\end{equation}
in the sense that nonnegative normalized eigenfunctions of \eqref{dirp} converge, up to a subsequence, to a positive Lipschitz function $v_\infty$ which solves   \eqref{dirinf} in the viscosity sense with $\lambda(\Omega)=\lambda_\infty(\Omega)$. Finally they also show that the infinity Laplacian eigenvalue problem \eqref{dirinf} admits nontrivial solutions if and only if $\lambda\ge\lambda_\infty$ and positive solutions if and only if $\lambda=\lambda_\infty$. Therefore they call $\lambda_\infty$ the  principal eigenvalue of the $\infty$-Laplacian  eigenvalue problem under Dirichlet boundary condition.\\

In the Neumann case (see \cite{RS}) and for any bounded connected $\Omega$ with Lipschitz boundary the limiting problem $p\to\infty$ for (\ref{neup}) is given by \eqref{infn}.
%\begin{equation}\label{mainsys}
%\begin{dcases}
%{\rm {min}}  \{| \grad u| - \Lambda_{\infty} u, -\Delta_{\infty} u\} = 0 & \mbox{in}\ \{u >0\}\cap\Omega\\
%{\rm {max}} \{  -| \grad u|  -\Lambda_{\infty} u, -\Delta_{\infty} u\} =0 & \mbox{in}\ \{u <0\}\cap\Omega\\
%-\Delta_{\infty} u=0,& \mbox{in} \ \{u =0\}\cap\Omega\\
%\frac{\partial u}{\partial \nu}=0 & \mbox{on} \ \partial\Omega.\\
%\end{dcases}
%\end{equation}

%Here the notation $\Lambda_\infty$ should not mislead the reader. 
In analogy to the Dirichlet case, the first nontrivial eigenvalues of (\ref{neup}) satisfy 
\begin{equation}\label{Lambda_limit}
\lim_{p\to\infty}\Lambda_p=\Lambda_\infty.
\end{equation}
Our result proves that on the class of convex sets the first nontrivial Neumann $p$-Laplacian eigenvalues converge to the first nontrivial Neumann $\infty$-Laplacian eigenvalue, namely $\Lambda=\Lambda_\infty$ is in fact the first nontrivial eigenvalue in \eqref{infn}.

Therefore we can point out some consequences.
\begin{Cor}
For convex $\Omega$ the first positive Neumann eigenvalue
$\Lambda_\infty(\Omega)$ is never larger than the first Dirichlet eigenvalue $\lambda_\infty(\Omega)$. Moreover $\lambda_\infty(\Omega)=\Lambda_\infty(\Omega)$ if and only if $\Omega$ is a ball.
\end{Cor}

The inequality $\Lambda_2(\Omega)<\lambda_2(\Omega)$   follows from a combination of the Szeg\"o-Weinberger and the Faber-Krahn inequalities, see e.g the books by Bandle or Kesavan \cite{Ba, K}. The strict inequality $\Lambda_p(\Omega)<\lambda_p(\Omega)$ for general $p$ and any convex $\Omega$ has been recently proved  in \cite{BNT} .

\begin{Cor}
For convex $\Omega$ any Neumann eigenfunction associated with $\Lambda_\infty(\Omega)$ cannot have a closed nodal domain inside $\Omega$.
\end{Cor}
Since a Neumann eigenfunction $u$ for the $\infty$-Laplacian is in general just continuous, a closed nodal line inside $\Omega$ means that there exists an opens subset $\Omega'\subset\Omega$ such that  $u> 0$ in $\Omega'$ (or $<0$ in $\Omega'$) and $u=0$ on $\partial\Omega'$. Assuming that such a nodal line exists, we can use standard arguments. We observe that $u$ is also a Dirichlet eigenfunction on $\Omega'$ with same eigenvalue. We get $\frac2{\diam(\Omega)}=\Lambda_\infty(\Omega)=\lambda_\infty(\Omega')=\frac1{R(\Omega')}\ge\frac2{\diam(\Omega)}$ and notice that the last inequality is strict for all sets other then balls. This proves the Corollary.

\medskip

%The last corollary that we want to mention is related the so-called Payne-Weinberger inequality. 
Next we recall that the Payne-Weinberger inequality states that on any convex subset $\Omega\subset \R^n$ the first non trivial Neumann eigenvalue for the Laplacian is bounded from below  by the quantity $\frac{\pi^2}{\diam(\Omega)^2}$. Recently such an estimate has been generalized to the first non trivial Neumann $p$-Laplacian eigenvalues in \cite{ENT,FNT,  V}  to get
\begin{equation}\label{PWg}
\Lambda_p\ge(p-1)^{1/p}\left(\frac{2\pi}{p\ \diam(\Omega)\sin\frac\pi p}\right).
\end{equation}
As $p\to\infty$ the right hand side in this Payne-Weinberger inequality \eqref{PWg} converges 
\begin{equation*}
\lim_{p\to\infty}%\Lambda_p^\frac1p\ge
(p-1)^{1/p}\left(\frac{2\pi}{p\ \diam(\Omega)\sin\frac\pi p}\right)=\frac{2}{\diam(\Omega)},
\end{equation*}
and in view of \eqref{Lambda_limit} we may therefore conclude that 
\begin{Cor}
The  Payne-Weinberger inequality \eqref{PWg} for the first Neumann eigenvalue of the $p$-Laplacian becomes an identity for $p=\infty$.
\end{Cor}

As a byproduct of our proofs we obtain also the following result, which is related to the hot-spot conjecture. The hot spot conjecture, see \cite{Bu}, says that a first nontrivial Neumann eigenfunction for the linear Laplace operator on a convex domain $\Omega$ should attain its maximum or minimum on the boundary $\partial\Omega$ and the proof of Lemma 1 will show that $u_\infty$ has this property as well. But there may be more than one eigenfunction associated to $\Lambda_\infty$. 

\begin{Cor}\label{hotspot}
If $\Omega$ is convex and smooth, then any first nontrivial Neumann eigenfunction, i.e. any viscosity solution to (\ref{infn}) for $\Lambda=\Lambda_\infty$ attains both its maximum and minimum only on the boundary $\partial\Omega$. Moreover the extrema of $u$ are located at points that have maximal distance in $\overline \Omega$.
\end{Cor}
The proof of our main result, Theorem \ref{thm_main}, will be a combination of Theorem \ref{main} in Section \ref{limsec} on the limiting problem as $p\to \infty$ and Proposition \ref{prop_eigen} in Section \ref{concase}. Corollary \ref{hotspot} will be derived at the very end of this paper.

\section{The limiting problem as $p\to \infty$}\label{limsec}

\begin{Def}\label{intrinsic}
Let $\Omega$ be a bounded open connected domain in ${\mathbb R}^n$.  The intrinsic diameter of $\Omega$, denoted by $\diam(\Omega)$, is defined as
\begin{equation}\label{intrdiam}
\textrm{diam}(\Omega):=\sup_{x,y\in \Omega}d_{\Omega}(x,y)
\end{equation}
 whith $d_{\Omega}$ denoting geodetic distance in $\Omega$.
\end{Def}
Consider the eigenvalue problem
\begin{equation}\Lambda_p^p=\min\left\{\frac{\int_{\Omega}|\grad v|^p\ dx}{\int_{\Omega}|v|^p\ dx}:\ v\in W^{1,p}(\Omega),\ \int_{\Omega} |v|^{p-2}v\ dx=0 \right\}.
\label{eigenvalue}
\end{equation}
Let $u_p$ be a  minimizer  of ($\ref{eigenvalue}$) such that $||u_p||_p=1$, where  $||f||^p_p=\frac{1}{|\Omega|}\int_{\Omega} \mid f\mid^p \ dx$.

For every $p>1$ $u_p$ satisfies the Euler equation
\begin{equation}
\label{euler}
\left\{
\begin{array}
{ll}-{\rm div} \Bigl(|\grad u_p|^{p-2}\grad u_p\Bigr)=\Lambda_p^p|u_p|^{p-2}u_p  &\mbox{ in } \Omega \\\\
 |\grad u_p|^{p-2}\frac{\partial u_p}{\partial \nu}=0 & \mbox{ on } \partial\Omega
\end{array}
\right.
\end{equation}
and 
\begin{Lem}
Let $\Omega$ be a connected bounded open set in $\R^n$ with Lipschitz boundary, then 
\begin{equation}
\lim_{p\rightarrow +\infty} \Lambda_p=\Lambda_{\infty}:=\frac{2}{\diam(\Omega)},
\end{equation}
here  $\diam(\Omega)$ denotes the  intrinsic diameter as defined  in (\ref{intrdiam}).
\end{Lem}
\begin{proof}
{{\bf Step 1} $\limsup_{p\rightarrow \infty} \Lambda_p \le \dfrac{2}{\diam(\Omega)}$ .}\\ We start proving that $\Lambda_{\infty}\leq 2/\diam(\Omega)$.
Let $x_0\in \Omega$. We choose $c_p\in \mathbb R$ such that $w(x)=d_{\Omega}(x,x_0)-c_p$ is a good test function in ($\ref{eigenvalue}$), that is $$\int_{\Omega} |w|^{p-2}w\ dx=0.$$ Using this test function in ($\ref{eigenvalue}$) we get (recalling that $| \grad  d_{\Omega}(x,x_0)| \le 1$ a.e. in $\Omega$)
\begin{equation}
\label{leql}
\Lambda_p\leq \frac{1}{\Bigl(\frac{1}{|\Omega|}\int_{\Omega}|d_{\Omega}(x,x_0)-c_p|^p\Bigr)^{1/p}}.
\end{equation}
Now we observe that $0 \leq c_p\leq \diam(\Omega)$ and thus up to a subsequence $c_p\rightarrow c$, with $0 \leq c\leq \diam(\Omega)$, then we obtain
$$
\liminf_{p\rightarrow\infty}\Bigl(\frac{1}{|\Omega|}\int_{\Omega}|d(x,x_0)-c_p|^p\Bigr)^{1/p}=\sup_{x\in \Omega}|d_{\Omega}(x,x_0)-c|\geq \diam(\Omega)/2
$$
and then from ($\ref{leql}$) the Step $1$ is proved.\\
{{\bf Step 2} {$\liminf_{p\rightarrow \infty} \Lambda_p \ge \dfrac{2}{\diam(\Omega)}$ .}}\\
By definition we get
$$
\Bigl(\frac{1}{|\Omega|}\int_{\Omega}|\grad u_p(x)|^p dx\Bigr)^{1/p} = \Lambda_p.
$$
Let us fix $m>n$. For $p>m$ by H\"older inequality we have
$$
\Bigl(\frac{1}{|\Omega|}\int_{\Omega}|\grad u_p(x)|^m dx\Bigr)^{1/m}\leq \Lambda_p.
$$
We can deduce that $\{u_p\}_{p\geq m}$ is uniformly bounded in $W^{1,m}(\Omega)$ and then assume that, up to a subsequence, $u_{p}$ converges weakly in $W^{1,m}(\Omega)$ and  in $C^0(\Omega)$ to a function $u_{\infty} \in W^{1,m}(\Omega)$. For $q>m$, by semicontinuity and H\"older inequality, we get 
$$
\frac{||\grad u_{\infty}||_q}{||u_{\infty}||_q}\leq \liminf_{p\rightarrow \infty} \frac{\Bigl(\frac{1}{|\Omega|}\int_{\Omega}|\grad u_p(x)|^q dx\Bigr)^{1/q}}{\Bigl(\frac{1}{|\Omega|}\int_{\Omega}|u_p(x)|^q dx\Bigr)^{1/q}}\leq\liminf_{p\rightarrow \infty} \frac{\Bigl(\frac{1}{|\Omega|}\int_{\Omega}|\grad u_p(x)|^p dx\Bigr)^{1/p}}{\Bigl(\frac{1}{|\Omega|}\int_{\Omega}|u_p(x)|^q dx\Bigr)^{1/q}}.
$$
Thus
\begin{equation}
\label{normaq}
\frac{||\grad u_{\infty}||_q}{||u_{\infty}||_q} \leq {\frac{||u_{\infty}||_{\infty}} {||u_{\infty}||_q}} \liminf_{p\rightarrow \infty} \Lambda_p
\end{equation}
letting $q \rightarrow \infty$ we get
\begin{equation}
\label{norm}
\frac{||\grad u_{\infty}||_{\infty}}{||u_{\infty}||_{\infty}} \leq  \liminf_{p\rightarrow \infty} \Lambda_p.
\end{equation}

Now we observe that condition $\int_{\Omega}|u_p|^{p-2}u_p=0$ leads to
\begin{equation}
\label{meaninf}
\sup u_{\infty} =- \inf u_{\infty},
\end{equation}
infact we have
\begin{equation}
\begin{array}{ll}
0 &\leq \Big| \|(u_{\infty})^{+}\|_{p-1} - \|(u_{\infty})^{-}\|_{p-1} \Big| \\
\cr
&= \Big| \|(u_{\infty})^{+}\|_{p-1}  - \|(u_{p})^+\|_{p-1} + \|(u_{p})^-\|_{p-1} - \|(u_{\infty})^{-}\|_{p-1} \Big| 
\\
\cr
&\le
\Big| \|(u_{\infty})^+\|_{p-1} - \|(u_{p})^+\|_{p-1} \Big| + \Big| \|(u_{\infty})^-\|_{p-1} - \|(u_{p})^-\|_{p-1} \Big|\\
\cr
&\leq \| (u_{\infty})^+ -  (u_{p})^+ \|_{p-1} + \| (u_{\infty})^- -  (u_{p})^- \|_{p-1}.
\end{array}
\end{equation}
Letting $p \rightarrow \infty$ we obtain \eqref{meaninf}.
Using the following inequality (see for instance \cite{B}, p.269)
$$
|u_{\infty}(x) - u_{\infty}(y)| \leq d_{\Omega}(x,y) ||\grad u_{\infty}||_{\infty} \le \diam{(\Omega)} ||\grad u_{\infty}||_{\infty},
$$
  we can conclude the proof  by \eqref{norm}
observing that
$$
2||u||_{\infty} = \sup u_{\infty} - \inf u_{\infty} \leq \diam(\Omega) ||\grad u_{\infty}||_{\infty}.
$$
\end{proof}
\begin{Rem}
Our proof shows that $u_\infty$ increases with constant slope $\Lambda_\infty||u_\infty||_\infty$ along the geodesic between two point spanning $diam(\Omega)$. In a rectangle this would be a diagonal.
\end{Rem}

Before proving Theorem \ref{main} we recall the definition of viscosity super (sub) solution to
\begin{equation}
\label{Neu}
\begin{dcases}
F(u, \grad u, \grad^2u)={\rm {min}}  \{| \grad u| - \Lambda |u|, -\Delta_{\infty} u\} = 0 & \mbox{in}\ \{u >0\}\cap\Omega\\
G(u,\grad u, \grad^2u)={\rm {max}} \{  \Lambda |u| -| \grad u|, -\Delta_{\infty} u\}=0 & \mbox{in}\ \{u <0\}\cap\Omega\\
H(\grad^2 u) =-\Delta_{\infty} u= 0,& \mbox{in} \ \{u =0\}\cap\Omega\\
\frac{\partial u}{\partial \nu}=0 & \mbox{on} \ \partial\Omega.\\
\end{dcases}
\end{equation}

\begin{Def}
\label{subs}
An upper semicontinuous  function $u$ is a viscosity subsolution to \eqref{Neu} if whenever $x_0 \in \Omega$ and $\phi \in C^2(\Omega)$ are such that
\[
u(x_0) = \phi(x_0), \quad {\mbox and} \>\> u(x)<\phi(x)  \>\>  {\mbox if \>}  x\neq x_0, \quad {\mbox then}
\] 
 \begin{equation}
\label{func}
F(\phi(x_0), \grad\phi(x_0), \grad^2\phi(x_0)) \le 0 \quad {\mbox if} \> u(x_0)>0
\end{equation}
\begin{equation}
\label{func2}
G(\phi(x_0), \grad\phi(x_0), \grad^2\phi(x_0)) \le 0 \quad {\mbox if} \> u(x_0)<0
\end{equation}
\begin{equation}
\label{func3}
H(\grad^2\phi(x_0)) \le 0 \quad {\mbox if} \> u(x_0)=0,
\end{equation}
while if $x_0 \in \partial \Omega$  and $\phi \in C^2(\bar \Omega)$  are such that 
\[
u(x_0) = \phi(x_0), \quad {\mbox and} \>\> u(x)<\phi(x)  \>\>  {\mbox if \>}  x\neq x_0, \quad {\mbox then}
\]
\begin{equation}
\label{BC}
{\rm {min}} \{F(\phi(x_0), \grad\phi(x_0), \grad^2\phi(x_0)), \frac{\partial \phi}{\partial{\nu}}(x_0) \} \le 0  \quad {\mbox if} \> u(x_0)>0
\end{equation}
\begin{equation}
\label{BC2}
{\rm {min}} \{G(\phi(x_0), \grad\phi(x_0), \grad^2\phi(x_0)), \frac{\partial \phi}{\partial{\nu}}(x_0) \} \le 0  \quad {\mbox if} \> u(x_0)<0
\end{equation}
\begin{equation}
\label{BC3}
{\rm {min}} \{H(\grad^2\phi(x_0)), \frac{\partial \phi}{\partial{\nu}}(x_0) \} \le 0  \quad {\mbox if} \> u(x_0)=0.
\end{equation}
\end{Def}

\begin{Def}
\label{supers}
A lower semicontinuous function $u$ is a viscosity supersolution to \eqref{Neu} if whenever $x_0 \in \Omega$ and $\phi \in C^2(\Omega)$ are such that
\[
u(x_0) = \phi(x_0), \quad {\mbox and} \>\> u(x)>\phi(x)  \>\>  {\mbox if \>}  x\neq x_0, \quad {\mbox then}
\] 
 \begin{equation}
\label{dfunc}
F(\phi(x_0), \grad\phi(x_0), \grad^2\phi(x_0)) \ge 0 \quad {\mbox if} \> u(x_0)>0
\end{equation}
\begin{equation}
\label{dfunc2}
G(\phi(x_0), \grad\phi(x_0), \grad^2\phi(x_0)) \ge 0 \quad {\mbox if} \> u(x_0)<0
\end{equation}
\begin{equation}
\label{dfunc3}
H(\grad^2\phi(x_0)) \ge 0 \quad {\mbox if} \> u(x_0)=0,
\end{equation}

\noindent
while if $x_0 \in \partial \Omega$  and $\phi \in C^2(\bar \Omega)$  are such that 
\[
u(x_0) = \phi(x_0), \quad {\mbox and} \>\> u(x)>\phi(x)  \>\>  {\mbox if \>}  x\neq x_0, \quad {\mbox then}
\]
then
\begin{equation}
\label{DBC}
{\rm {max}} \{F(\phi(x_0), \grad\phi(x_0), \grad^2\phi(x_0)), \frac{\partial \phi}{\partial{\nu}}(x_0) \} \ge 0  \quad {\mbox if} \> u(x_0)>0
\end{equation}
\begin{equation}
\label{DBC2}
{\rm {max}} \{G(\phi(x_0), \grad\phi(x_0), \grad^2\phi(x_0)), \frac{\partial \phi}{\partial{\nu}}(x_0) \} \ge 0  \quad {\mbox if} \> u(x_0)<0
\end{equation}
\begin{equation}
\label{DBC3}
{\rm {max}} \{H(\grad^2\phi(x_0)), \frac{\partial \phi}{\partial{\nu}}(x_0) \} \ge 0  \quad {\mbox if} \> u(x_0)=0.
\end{equation}
\end{Def}

\begin{Def}
\label{sol}
A continuous function $u$ is a solution to \eqref{Neu} iff it is both a supersolution and a subsolution to \eqref{Neu}
\end{Def}

\begin{Rem}
It is instructive to use the definition for checking that the one-dimensional function $u(x)=x_1$ on the square $\Omega=(-1,1)\times(-1,1)$ is a viscosity solution of (\ref{Neu}). In fact,
$u\in C^2(\Omega)$, and $-\Delta_\infty u=0$ in $\Omega$. 

So the first PDE in (\ref{Neu}) is satisfied  if also $1=|\nabla u|\geq\Lambda u$ on $\{ u>0\}$, and that implies $\Lambda\leq 1.$

The Neumann boundary condition is satisfied in classical sense on horizontal parts of $\partial\Omega$. However, for Neumann condition to hold in the viscosity sense on the right part,  
we must verify
$$\min \{ \min \{| \nabla \phi|-\Lambda \phi, -\Delta_\infty\phi\}\ , \ \partial\phi/\partial \nu\}(x_0)\leq 0$$
for any $C^2$ test function $\phi$ touching $u$ in $x_0\in\partial\Omega$ from above, and
$$\max\{ \min \{|\nabla \psi|-\Lambda \psi, -\Delta_\infty\psi\}\ , \ \partial\psi/\partial \nu\}(x_0)\geq 0$$
for any smooth test function $\psi$ touching $u$ from below. 

Recall $|\nabla u|=\partial u/\partial\nu=1$ everywhere. Therefore only the very first constraint is active on the boundary and implies $$\Lambda\geq 1.$$
This shows that $u(x)=x_1$ is a viscosity solution to (\ref{Neu}) with eigenvalue $\Lambda=1$, but $$\Lambda=1>\frac{1}{\sqrt{2}}=\frac{2}{diam(\Omega)}=\Lambda_\infty.$$

\end{Rem}

In what follows we will use the notation
\[
F_p(u, \grad u, \grad^2 u) = -(p-2) |\grad u|^{p-4} \Delta_{\infty} u - |\grad u|^{p-2} \Delta u - \Lambda_p^p |u|^{p-2}u
\]
with
\[ 
\Delta_{\infty} u =\sum_{i,j=1}^n u_{x_i}u_{x_ix_j}u_{x_j}.
\]

\begin{Lem}\label{weakvisc}
Let $u \in W^{1,p}(\Omega)$ be a weak solution to 
\begin{equation}
\left\{
\begin{array}
{ll}-{\rm div} \Bigl(|\grad u|^{p-2}\grad u\Bigr)=\Lambda_p^p|u|^{p-2}u  &\mbox{ in } \Omega \\\\
 |\grad u|^{p-2}\frac{\partial u}{\partial \nu}=0 & \mbox{ on } \partial\Omega
\end{array}
\right.
\end{equation}
then $u$ is a viscosity solution to
\begin{equation}
\label{plapvisc}
\left\{
\begin{array}
{ll} F_p(u,\grad u,\grad^2u) =0 &\mbox{ in } \Omega \\\\
 |\grad u|^{p-2}\frac{\partial u}{\partial \nu}=0 & \mbox{ on } \partial\Omega.
\end{array}
\right.
\end{equation}
\end{Lem}
\begin{proof}
That $u$ is a viscosity solution to the differential equation $F_p=0$ in $\Omega$ was shown in \cite{JLM}, Lemma 1.8. It remains to show that the Neumann boundary condition is satisfied in the viscosity sense as defined for instance in    
\cite{GMPR}.
Let $x_0 \in \partial \Omega$, $\phi \in C^2(\bar \Omega)$ such that $u(x_0)= \phi(x_0)$ and $\phi(x) < u(x)$ when $x \neq x_0$. Assume by contradiction that
\begin{equation}
\label{plaplbd}
\max \{|\grad\phi (x_0)|^{p-2}  \frac{\partial \phi}{\partial \nu}(x_0), F_p( \phi(x_0), \grad \phi (x_0), \grad^2 \phi (x_0))  \} <0.
\end{equation}
Then there exists a ball $B_r(x_0)$, centered at $x_0$ with radius $r>0$, such that \eqref{plaplbd} holds true $\forall x \in \bar \Omega \cap B(x_0,r)$. Denoted by
$0 <m = \inf_{\bar \Omega \cap B_r(x_0)} (u(x)-\phi(x)) $ and by $\psi(x) = \phi(x) + \dfrac{m}{2}$.
Using $(\psi -u)^+$ as test function in the weak formulation we have both
\[
\int_{\psi >u} |\grad \psi|^{p-2}\grad \psi \grad(\psi -u) \,dx < \Lambda_p^p \int_{\psi >u}  |\phi|^{p-2}\phi (\psi -u) \, dx
\]
and
\[
\int_{\psi >u} |\grad u|^{p-2}\grad u \grad(\psi -u) \,dx = \Lambda_p^p \int_{\psi >u}  |u|^{p-2}u(\psi -u) \, dx\ .
\]
Subtraction yields the contradiction
\begin{equation}
\begin{array}{rl}
 C \int_{\psi >u} |\grad (\psi -u)|^p \,dx & \leq \int_{\psi >u} \left(|\grad \psi|^{p-2}\grad \psi - |\grad u|^{p-2}\grad u, \grad(\psi -u)\right) \, dx  \\
\cr
&<\Lambda_p^p \int_{\psi >u}  (|\phi|^{p-2}\phi -  |u|^{p-2}u)(\psi -u) \, dx <0.
\end{array}
\end{equation}
\end{proof}

%The equation and the BC (boundary condition) both satisfied in the viscosity sense by $\Lambda_\infty$ and $u_\infty$ are the following
%\begin{equation}
\begin{Thm}\label{main}
Let $\Omega$ be an open bounded connected set of $\R^n$. If $u_\infty$ and $\Lambda_\infty$ are defined as above then $u_\infty$ satisfies \eqref{Neu} in the viscosity sense %the following eigenvalue problem 
with $\Lambda=\Lambda_\infty$.
%\begin{eqnarray}
%F(u, \grad^2u)= {\rm {min}}  \{| \grad u| - \Lambda |u|, -\Delta_{\infty} u\} = 0 \quad {\mbox {in}} \> \Omega \cap \{u>0\}
%\\ \label{min}
%\cr
%G(u, \grad^2u)={\rm {max}} \{  \Lambda |u| -| \grad u|, -\Delta_{\infty} u\}=0 \quad {\mbox {in}} \> \Omega \cap \{u<0\}\\ \label{max}
%\cr
%-\Delta_{\infty} u=0 \quad {\mbox {in}} \> \Omega \cap \{u=0\}\\ \label{equality}
%\cr
%\frac{\partial u}{\partial{\nu}} =0 \quad {\mbox {on}} \> \partial \Omega \label{BC}
%\end{eqnarray}
%%\end{equation}
\end{Thm}
%The boundary condition \eqref{BC} means that if $x_0 \in \partial \Omega$, $u(x_0)>0$, $\phi \in C^2(\bar \Omega)$ with
%$$
%\phi<u  \quad {\mbox {in}} \> \bar \Omega, \> u(x_0)= \phi(x_0)
%$$
%
%then
%$$
%{\rm {max}} \{F(\phi(x_0), \grad^2\phi(x_0)), \frac{\partial \phi}{\partial{\nu}}(x_0) \} \ge 0
%$$
%
%while 
%$$
%\phi>u  \quad {\mbox {in}} \> \bar \Omega, \> u(x_0)= \phi(x_0)
%$$
%
%then
%$$
%{\rm {min}} \{F(\phi(x_0), \grad^2\phi(x_0)), \frac{\partial \phi}{\partial{\nu}}(x_0) \} \le 0
%$$
%\smallskip
%
% if $x_1 \in \partial \Omega$, $u(x_1)<0$, $\psi \in C^2(\bar \Omega)$ with
% 
% $$
%\psi<u  \quad {\mbox {in}} \> \bar \Omega, \>   \phi(x_1) =u(x_1)
%$$
%
%then
%$$
%{\rm {max}} \{G(\psi(x_0), \grad^2\psi(x_0)), \frac{\partial \psi}{\partial{\nu}}(x_0) \} \ge 0
%$$
%
%$$
%\psi>u  \quad {\mbox {in}} \> \bar \Omega, \>   \phi(x_1) =u(x_1)
%$$
%
%then
%$$
%{\rm {min}} \{G(\psi(x_0), \grad^2\psi(x_0)), \frac{\partial \psi}{\partial{\nu}}(x_0) \} \le 0
%$$

\begin{proof} 
First we observe that in fact there exists a subsequence $u_{p_i}$  uniformly converging to $u_{\infty}$ in $\Omega$.
Now let us prove that $u_{\infty}$ is a viscosity super solution to \eqref{Neu} in $\Omega$.
Let $x_0 \in \Omega$ and let $\phi \in C^2(\Omega)$ be such that $ \phi (x_0)= u_{\infty}(x_0)$ and $ \phi (x)< u_{\infty}(x) \quad x \in \Omega \setminus \{x_0\}$.
Since $u_{p_i} \rightarrow u_\infty$ uniformly in $B_r(x_0)$ one can prove  that $u_{p_i} -\phi$ has a local minimum in $x_i$, with $\lim_i x_i =x_0$.
Recalling that $u_{p_i}$ is a viscosity solution to \eqref{plapvisc}, choosing  $\psi(x) = \phi(x)- \phi(x_i)+u_{p_i}(x_i)$ as test function we obtain
\begin{equation}
\label{F_p}
-[(p_i-2) |\grad\phi(x_i)|^{p_i-4} \Delta_{\infty} \phi(x_i)+ |\grad\phi(x_i)|^{p_i-2} \Delta \phi(x_i)] \ge \Lambda_{p_i}^{p_i }|u_{p_i}(x_i)|^{p_i-2}u_{p_i}(x_i).
\end{equation}
Three cases can occur.
\begin{itemize}
\item $u_{\infty}(x_0)>0.$ In this case \eqref{F_p} implies that $|\grad\phi(x_i)|>0$, hence dividing \eqref{F_p} by $|\grad\phi(x_i)|^{p_i-4}(p_i-2)$ we have 
\begin{equation}
\label{cc}
-\frac{|\grad\phi(x_i)|^2 \Delta \phi(x_i)}{p_i-2} - \Delta_{\infty} \phi(x_i) \ge \left (\frac{ \Lambda_{p_i} u_{p_i}(x_i)}{|\grad\phi(x_i)|} \right)^{p_i-4} \frac{\Lambda^4_{p_i}u^3_{p_i}(x_i)}{p_i-2}.
\end{equation}
Letting $ {p_i}$ go to $+\infty$ we have
 $\dfrac{\Lambda_{\infty} \phi(x_0)}{|\grad\phi(x_0)|} \le 1$ and $- \Delta_{\infty} \phi(x_0) \ge 0$ hence 
 $${\rm {min}}  \{| \grad \phi(x_0)| - \Lambda_{\infty} |\phi(x_0)|, -\Delta_{\infty} \phi(x_0)\} \ge 0. $$
 \item $u_{\infty}(x_0)<0.$ Also in this case \eqref{F_p} implies that $|\grad\phi(x_i)|>0$, and dividing by $|\grad\phi(x_i)|^{p_i-4}(p_i-2)$ we have again \eqref{cc}.
 If $\dfrac{\Lambda_{\infty} \phi(x_0)}{|\grad\phi(x_0)|} <  1$, letting $p_i$ go to $\infty$, we have $- \Delta_{\infty} \phi(x_0) \ge 0$, otherwise $\dfrac{\Lambda_{\infty} \phi(x_0)}{|\grad\phi(x_0)|} \ge  1$.
 In both cases we have
$$ {\rm {max}} \{  \Lambda_{\infty} |\phi(x_0)| -| \grad \phi(x_0)|, -\Delta_{\infty} \phi(x_0)\}\ge0.$$
\item $u_{\infty}(x_0) =0.$ If $|\grad \phi(x_0)|=0$ then, by definition, we have   $-\Delta_{\infty} \phi(x_0)=0$. If $|\grad \phi(x_0)|>0$ then
$
\lim_i {\dfrac{\Lambda_{p_i} |u_{p_i}(x_i)|}{|\grad \phi(x_i)| }=0}
$
hence \eqref{cc} implies 
$$
-\Delta_{\infty} \phi(x_0)\ge0.
$$
\end{itemize}

It remains to prove that $u_\infty$ satisfies the boundary conditions in the viscosity sense.

Assume that  $x_0 \in \partial\Omega$ and let $\phi \in C^2(\bar \Omega)$ be such that $ \phi (x_0)= u_{\infty}(x_0)$ and $ \phi (x)< u_{\infty}(x) \quad x \in \bar \Omega \setminus \{x_0\}$.
Using again the uniform convergence of $u_{p_i}$ to $u_{\infty}$ we  obtain that  $u_{p_i} -\phi$ has  a minimum point  $x_i \in \bar \Omega$,   with $\lim_i x_i =x_0$.

If $x_i \in \Omega$ for infinitely many $i$ arguing as before we get
$${\rm {min}}  \{| \grad \phi(x_0)| - \Lambda_{\infty} |\phi(x_0)|, -\Delta_{\infty} \phi(x_0)\} \ge 0, \quad \mbox{if} \> u(x_0) >0,$$
$$ {\rm {max}} \{  \Lambda_{\infty} |\phi(x_0)| -| \grad \phi(x_0)|, -\Delta_{\infty} \phi(x_0)\}\ge0, \quad \mbox{if} \> u(x_0) <0,$$
$$
-\Delta_{\infty} \phi(x_0)\ge0, \quad \mbox{if} \> u(x_0) =0.$$

If $x_i \in \partial \Omega$, since ${u_{p_i}}$ is viscosity solution to \eqref{plapvisc},  for infinitely many $i$ we have 
$$
|\grad \phi (x_i)|^{p_i-2} \frac{\partial \phi}{\partial \nu}(x_i)\ge 0
$$
which concludes the proof.

Arguing in the same way we can prove that $u_\infty$ is a viscosity subsolution to \eqref{Neu} in $\Omega$.

\end{proof}

\section{$\Lambda_\infty$ is the first non trivial eigenvalue}\label{concase}

\begin{Prop}\label{prop_eigen}
Let $\Omega $ be a smooth bounded open convex set in $\R^n$. If for some $\Lambda>0$ problem \eqref{Neu} admits a nontrivial eigenfunction $u$, then $\Lambda \ge \Lambda_\infty$. 
\end{Prop}

The main idea is to use a test function involving the distance from a suitable point $x_0\in\Omega$. This function is smooth everywhere except $x_0$. For the nonconvex case one may want to use intrinsic distance instead, which however is not of class $C^2$, as pointed out in \cite{ABB}.

%The proof is the result of the following intermediate steps ({\bf{only two of those steps are completed. In particular the sign change}}).

\begin{Lem}\label{no_min_inside}
Let $\Omega$, $\Lambda$ and $u$ be as in the statement of Proposition \ref{prop_eigen}. Let $\Omega_1$ be an open connected subset of $\Omega$ such that $u\ge m$ in $\bar\Omega_1$ for some positive constant $m$. Then $u>m$ in $\Omega_1$.   
\end{Lem}

\begin{proof}
%Since $u$ is a nontrivial solutions, we can always assume that $u$ is positive somewhere (if not we can always change its sign). 
Let $x_0$ be any point in $\Omega_1$. Our aim is to show that $u(x_0)>m$. Obviously, for any given $R>0$ such that $B_R(x_0)\subset \Omega_1$ we have $u\not\equiv m$ in $B_R(x_0)$ otherwise we have in $B_R(x_0)$ that $|\grad u|-\Lambda |u|< 0$ (in the viscosity sense) which violates the first equation in \eqref{Neu}. This means that for any $R>0$ such that $B_R(x_0)\subset \Omega_1$ it is possible to find $x_1\in B_{R/4}(x_0)$  such that $u(x_1)>m$. The continuity of $u$ implies that for some $\varepsilon>0$ small enough, there exists $r\le \textrm{dist}(x_0,x_1)$ such that $u>m+\varepsilon$ on $\partial B_r(x_1)$. Therefore the function

$$v(x)=m+\frac{\varepsilon}{\frac{R}{2}-r}\left(\frac{R}{2}-|x-x_1|\right) \qquad \mbox{in $B_{R/2}(x_1)\setminus B_r(x_1)$}$$
is such that $$- \Delta_\infty v=0\qquad \mbox{in $B_{R/2}(x_1)\setminus B_r(x_1)$}.$$
Since
$$-\Delta_\infty u \ge 0\qquad \mbox{in $B_{R/2}(x_1)\setminus B_r(x_1)$}$$ in the viscosity sense, and
$$u\ge v \qquad \mbox{on $\partial B_{R/2}(x_1)\cup \partial B_r(x_1)$}$$
the comparison principle, see Theorem 2.1 in \cite{Je}, implies that $u\ge v > m$ in $B_{R/2}(x_1)\setminus B_r(x_1)$ and therefore $u(x_0)> m$.
\end{proof}

\begin{Lem}
\label{sign}
Let $\Omega$, $\Lambda$ and $u$ be as in the statement of Proposition \ref{prop_eigen}. Then $u$ certainly changes sign.
\end{Lem}
\begin{proof}
Since $u$ is a nontrivial solution to \eqref{Neu}, we can always assume, possibly changing the sign of the eigenfunction $u$, that it is positive somewhere. We shall prove that the minimum of $u$ in $\bar\Omega$ is negative. We argue by contradiction and we assume that the minimum $m$ is nonnegative. In view of Lemma \ref{no_min_inside} a positive minimum can not be attained in $\Omega$. On the other hand zero as well can not be attained as minimum in $\Omega$. If so, since $u\not\equiv 0$, there would exist a point $x_0\in\Omega$ and a ball $B_R(x_0)\subset \Omega$ such that $u(x_0)=0$ and $\max_{B_{R/4}(x_0)}u >0$. %Using a function $v$ constructed in a similar way as Lemma \ref{no_min_inside}. More precisely, l
Let $x_1\in B_{R/4}(x_0)$ be such that $u(x_1)>0$. The continuity of $u$ implies that there exists $r\le \textrm{dist}(x_0,x_1)$ such that $u>u(x_1)/2$ on $\partial B_r(x_1)$. Therefore the function

$$v(x)=\frac{u(x_1)}{R-2r}\left(\frac{R}{2}-|x-x_1|\right) \qquad \mbox{in $B_{R/2}(x_1)\setminus B_r(x_1)$}$$
is such that $$- \Delta_\infty v=0\qquad \mbox{in $B_{R/2}(x_1)\setminus B_r(x_1)$}.$$
Since
$$-\Delta_\infty u \ge 0 \qquad \mbox{in $B_{R/2}(x_1)\setminus B_r(x_1)$}$$ in the viscosity sense, and
$$u\ge v \qquad \mbox{on $\partial B_{R/2}(x_1)\cup \partial B_r(x_1)$}$$
the comparison principle, see Theorem 2.1 in \cite{Je}, implies that $u\ge v > 0$ in $B_{R/2}(x_1)\setminus B_r(x_1)$ and therefore $u(x_0)> 0$.

Therefore the only possibility is that there exists $x_0\in\partial\Omega$ nonnegative minimum point of $u$. We shall prove that $\frac{\partial u}{\partial \nu}(x_0)<0$ in the viscosity sense in contradiction to \eqref{dfunc}-\eqref{dfunc3}. Indeed there certainly exist $\bar x\in\Omega$ and $r>0$ such that the ball $B_{r}(\bar x)\subset \Omega$ is inner tangential to $\partial \Omega$ at $x_0$ and $\partial B_r(\bar x)\cap\partial \Omega=\{x_0\}$. Then the function

$$v(x)=u(\bar x) - \left(\frac{u(\bar x) - u(x_0)}{r}\right) \left(|x-\bar x|\right) \qquad \mbox{in $B_r(\bar x)\setminus \left\{\bar x\right\}$}$$
satisfies 
$$- \Delta_\infty v=0\qquad \mbox{in $B_r(\bar x)\setminus \left\{\bar x\right\}$}$$
since $$-\Delta_\infty u \ge 0\qquad \mbox{in $B_r(\bar x)\setminus \left\{\bar x\right\}$}$$ in the viscosity sense, 
and
$$u\ge v \qquad \mbox{on $\partial B_r(\bar x)\cup\left\{\bar x\right\}$}.$$
Using again  the comparison principle, see Theorem 2.1 in \cite{Je},  we get $u\ge v$ in $\bar\Omega$. 
Therefore the function 
%$$\phi=\displaystyle\frac{u(\bar x)-m}{2r^2}(r^2-|x-\bar x|^2)+m$$
$$\phi=u(\bar x) - \left(u(\bar x) - u(x_0)\right)\left( \frac{|x-\bar x|}{r}\right)^{\frac{1}{2}}$$
is such that $\phi \in C^2(\bar \Omega-\{\bar x\}),$
$$
\phi< v\le u  \quad {\mbox {in}} \> B_r(\bar x)-\{\bar x\}, 
$$
$$
\phi(x)<u(x_0)\le u(x)\quad {\mbox {in}} \> \Omega\setminus B_r(\bar x),
$$
and
$$\> u(x_0)= \phi(x_0).$$
However
\begin{equation}\label{violation}
{\rm {max}} \{F(\phi(x_0),  \grad\phi(x_0),\grad^2\phi(x_0)), \frac{\partial \phi}{\partial{\nu}}(x_0) \} < 0
\end{equation}
%which 
contradicts \eqref{dfunc}-\eqref{dfunc3}. %provided the function $\phi$ is properly smoothed in a neighborhood of ${\bar x}$ so that  \eqref{violation} is still verified and the function $\phi-u$ on $\bar\Omega$ has still a minimum point in $x_0$.
\end{proof}

%This conjecture is the ultimate goal. This might require the following additional steps.
%\begin{itemize}
%\item Prove that $u$ changes sign
%\item Prove that $\sup u = -\inf u$
%\item Prove that $u$ achieves maximum on the boundary
%\end{itemize}

\begin{proof}[Proof of Proposition \ref{prop_eigen}]

Let $u$ be a non trivial eigenfunction of \eqref{Neu} and let us denote by $\Omega_+ = \{x \in \Omega : u(x) > 0  \}$ and by $\Omega_- = \{x \in \Omega : u(x) < 0  \} $.  Lemma \ref{sign} ensures that they are both nonempty sets.
Let us normalize the eigenfunction $u$ such that $${\max_{\bar\Omega}} \> u = \frac{1} {\Lambda}.$$ Then $\Lambda u \le 1$ which implies that 
\begin{equation}\label{ineqforu}
{\rm {min}}  \{| \grad u| - 1, -\Delta_{\infty} u\} \le 0 \quad {\mbox {in}} \> \Omega_+
\end{equation}
in the viscosity sense. 

For every $x_0 \in \Omega \setminus \Omega_+$ and for every $\epsilon>0$ and $\gamma >0$ the function  $g_{\epsilon,\gamma}(x) = (1+\epsilon)|x-x_0| -\gamma |x-x_0|^2$ belongs to
$C^2(\Omega \setminus B_{\rho}(x_0))$ for every $\rho>0$. If $\gamma$ is small enough compared to $\epsilon$, it verifies
\begin{equation}\label{ineqforg}
{\rm {min}}  \{| \grad g_{\epsilon,\gamma}| - 1, -\Delta_{\infty} g_{\epsilon,\gamma}\} \ge 0 \quad {\mbox {in}} \> \Omega_+.
\end{equation}
Therefore (a comparison) Theorem 2.1 in \cite{Je} ensures that
\begin{equation}\label{comp}
m= \inf_{x \in \Omega_+} (g_{\epsilon,\gamma}(x)-u(x)) = \inf_{x \in \partial \Omega_+} (g_{\epsilon,\gamma}(x)-u(x)).
\end{equation}
Now $\partial\Omega_+$ contains certainly points in $\Omega$ and possibly on $\partial\Omega$. To rule out that the infimum in the right hand side of (\ref{comp}) is attained on $\partial\Omega$, assume that there exists $\bar x \in  \partial \Omega \cap \partial \Omega_+$ such that $g_{\epsilon,\gamma}(\bar x) -u(\bar x) =m$
and choose $g_{\epsilon,\gamma}-m$ as test function in \eqref{BC}.
By construction for every $x \in \partial \Omega \cap \partial \Omega_+$ and $\gamma < \frac{\epsilon}{2 \diam(\Omega)}$ it results that
\[
|\grad g_{\epsilon,\gamma}|(x) =1+\epsilon -2\gamma |x-x_0|>1, 
\]
\[
\quad {\frac{\partial g_{\epsilon,\gamma}}{\partial \nu} (x)} = ((1+\epsilon) -2\gamma |x-x_0|) \left({\frac{x-x_0}{|x-x_0|}}, \nu(x)\right)>0, 
\]
and $$\quad -\Delta_{\infty} g_{\epsilon,\gamma}= 2 \gamma|\grad g_{\epsilon,\gamma}|^2>0$$ which give  a contradiction to \eqref{BC}. 
Together with \eqref{comp} this implies that
\[
m= \inf_{x \in \Omega_+} (g_{\epsilon,\gamma}(x)-u(x)) =\inf_{x \in \partial \Omega_+\cap \Omega} (g_{\epsilon,\gamma}(x)-u(x)) \ge0\ .
\]

Letting $\epsilon$ and $\gamma$ go to zero we have that
\begin{equation}\label{Dist}
|x-x_0| \geq u(x) \quad \forall x \in \{y: u(y) \ge 0\}, \quad \forall x_0 \in \{ y : u(y) \le 0\}
\end{equation}
hence
\[
d^+ = \sup_{x \in \bar \Omega_+} {\rm dist}{(x, \{u=0\})} \ge {\frac{1}{\Lambda}}.
\]

Arguing in the same way we obtain

\[
d^- = \sup_{x \in \bar \Omega_-} {\rm dist}{(x, \{u=0\})} \ge {\frac{1}{\Lambda}}
\]
hence
\[
\diam(\Omega)  \ge d^++d^- \ge {\frac{2}{\Lambda}}
\]
which concludes the proof of our proposition.
\end{proof}
Corollary \ref{hotspot} follows now easily. Returning to (\ref{Dist}) pick $x=\overline x$ as the point in which $u$ attains its maximum and correspondingly $x=\underline x$ as the point in which $u$ attains its minimum. Then
$d({\overline x}, \Omega_-)\geq \tfrac{1}{\lambda}$ and $d({\underline x}, \Omega_+)\geq \tfrac{1}{\lambda}$, so that $diam(\Omega)\geq |{\overline x} -{\underline x}|\geq \tfrac{2}{\Lambda}$. Since $\Lambda=\Lambda_\infty$, equality holds and the max and min of $u$ are attained in boundary points which have farthest distance from each other.

\end{document}